\documentclass[12pt,a4paper]{amsart}
\usepackage{amsfonts,color}
\usepackage{amsthm}
\usepackage{amsmath}
\usepackage{amscd}
\usepackage{amssymb}
\usepackage[latin2]{inputenc}
\usepackage{t1enc}
\usepackage[mathscr]{eucal}
\usepackage{indentfirst}
\usepackage[numbers,sort]{natbib}
\usepackage{graphicx}
\usepackage{graphics}
\usepackage{float}
\usepackage{pict2e}
\usepackage{epic}
\usepackage{url}
\usepackage{epstopdf}
\usepackage{comment}
\usepackage{bm}
\usepackage{array}
\usepackage{wrapfig}
\usepackage{multirow}
\usepackage[T1]{fontenc}
\usepackage{dsfont}
\usepackage{enumerate} 
\usepackage{tabularx}
\usepackage{relsize}

\usepackage{hyperref}
\numberwithin{equation}{section}
\usepackage[margin=3cm]{geometry}

\allowdisplaybreaks

\theoremstyle{plain}
\newtheorem{theorem}{Theorem}[section]
\newtheorem{lemma}[theorem]{Lemma}
\newtheorem{corollary}[theorem]{Corollary}
\newtheorem{proposition}[theorem]{Proposition}

 \theoremstyle{definition}
\newtheorem{definition}[theorem]{Definition}

\newtheorem{remark}[theorem]{Remark}
\newtheorem{claim}[theorem]{Claim}
\newtheorem{assumption}[theorem]{Assumption}

\newcommand{\gt}{g}

\newcommand{\kk}{\alpha}
\newcommand{\Ct}{C}

\begin{document}

\title{Permuton limit of a generalization of the Mallows and $k$-card-minimum models}

\author{Joanna Jasi\'nska}

\author{Bal\'azs R\'ath}

\address{ ELTE: E\"{o}tv\"{o}s Lor\'{a}nd University Mathematics Institute, H-1117 Budapest, P\'azm\'any P\'eter s\'et\'any 1/C. HUN-REN Alfr\'ed R\'enyi Institute of Mathematics, Re\'altanoda utca 13-15, H-1053 Budapest, Hungary.}
\email{joanna.jasinska1998@gmail.com}

\address{ Department of Stochastics,
Institute of Mathematics,
Budapest University of Technology and Economics,
M\H{u}egyetem rkp. 3., H-1111 Budapest, Hungary. HUN-REN Alfr\'ed R\'enyi Institute of Mathematics, Re\'altanoda utca 13-15,
H-1053 Budapest, Hungary.}
\email{rathbalazs@gmail.com}

\thanks{
}



\begin{abstract} We introduce and study a new random permutation model that generalizes the $k$-card minimum model \cite{Travers} and the Mallows model \cite{Mallows}. 
We calculate the permuton limit of such a sequence of random permutations. As a corollary, we deduce the law of large numbers for pattern densities. Moreover, we prove a universality result about the band structure of the limiting permuton, confirming a conjecture of Travers about the $k$-card minimum model. More specifically, we show that if  a certain model parameter goes to infinity then the appropriately scaled restriction of the permuton measure  to a line that intersects the diagonal perpendicularly  converges weakly to the logistic distribution.

\end{abstract}

\maketitle

\section{Introduction}

A permutation of $[n]=\{1,2,\ldots, n\}$ can be obtained by sequentially picking cards from the deck $[n]$. In \cite{Travers} Travers described the $k$-card-minimum model (or $k$CM for short), where in every step one chooses $k$ cards independently and uniformly at random  from the remaining deck and then picks the lowest one to be the next card in the rearranged deck. This method produces permutations which (in some sense) get closer and closer to the identity permutation as we make the parameter $k$ larger and larger. Another well-known example of a random permutation model is the Mallows model, introduced by Mallows in \cite{Mallows}. There one chooses a given permutation $\pi$ with probability proportional to $q^{\mathrm{inv}(\pi)}$ where $q$ is a positive real parameter and $\mathrm{inv}(\pi)$ denotes the number of inversions in $\pi$. Again, in the parameter regime where inversions are more and more penalized, the random permutation will get closer and closer to the identity permutation.
In \cite[Section 6]{Travers} Travers conjectures that these models have a similar band structure around the diagonal when $q \approx 1-\frac{k}{n}$ and $k$ is large, noting that there is a ``card-picking'' algorithm (very similar to the $k$CM model, to be described in Section \ref{section_gen_mod}) that generates a random permutation with the  same law as the Mallows model.

Motivated by the above two models, we introduce and study a model that generalizes both of them. Again, we sequentially remove cards from the deck, and the model is characterised by the probability distribution with which we pick the next card from the remaining deck. Under appropriate assumptions on this distribution, we prove in Theorem \ref{permuton_convergence_thm} that the sequence of random permutations that we generate converges in probability to a deterministic permuton \cite{the_paper} as $n \to \infty$, and hence we obtain weak law of large numbers for pattern densities as a corollary. Our result generalizes that of \cite{starr} about the permuton limit of the Mallows model.

We also show that (under certain natural assumptions) our general model exhibits a band structure around the diagonal. More specifically, in Theorem \ref{logistic_limit_thm} we derive a limit theorem for the (appropriately scaled) distribution of mass of the permuton around the diagonal as a certain model parameter goes to infinity. The limiting distribution (the so-called logistic distribution) is the same for all models that satisfy the conditions of Theorem \ref{logistic_limit_thm} (which includes the Mallows and $k$CM models),
confirming the conjecture of Travers formulated in \cite[Section 6]{Travers} about the universality of the band structure of these random permutation models.

Our Theorem \ref{permuton_convergence_thm} fits in the ongoing endeavor to identify the permuton limits of natural random permutation models. Other papers in this direction include \cite{runsort, separable, biased_permutations}.

\section{Statements of main results}\label{section_gen_mod}

We denote by $S_n$ the set of permutations of $[n]$.

\begin{definition}[$\gt$-random permutation]\label{def_gen_perm_model} Assume given a strictly decreasing $C^1[0,1]$ function $\gt$ satisfying $\gt(0)=1$ and $\gt(1)=0$.
For each $n \in \mathbb{N}_+$ and $i \in [n]$ let $\nu^{(n),i}_{\gt}$ denote the probability distribution supported on $[n-(i-1)]$ with probability masses
\begin{equation}\label{nu_prob_def_eq}
    \nu^{(n),i}_{\gt}( \{ l \} ):= \frac{\gt\left(\frac{(i-1)+(l-1)}{n}\right)-\gt\left(\frac{(i-1)+l}{n}\right)  }{\gt\left( \frac{i-1}{n} \right)}, \qquad l=1,\dots, n-(i-1).
\end{equation}
We say that a random element $\sigma_n$ of $S_n$ is a $\gt$-random permutation, or briefly write $\sigma_n \sim \mathrm{PERM}(\gt,n)$, if it can be generated as follows. We write the number $i$ on the $i$'th card of the original deck (counted from the bottom up).
$\sigma_n(i)$ will denote the number written on the card that is at position  $i$  in the new deck (counted from the bottom up). We build the new deck by picking a card from the remaining old deck  in each step and placing it at the top of the new deck.
\begin{itemize} \item  Let $\Ct^{(n),i}, \, i \in [n]$ denote independent random variables with distribution \begin{equation*} \Ct^{(n),i} \sim \nu^{(n),i}_{\gt}.\end{equation*}
\item For any $i \in [n]$ we will denote by $D^{(n)}_i$ the set of cards remaining in the old deck before the $i$'th step. Note that $|D^{(n)}_i|=n-(i-1)$.
\item
Let us define $D^{(n)}_i$   and $\sigma_n(i)$ for $i \in [n]$  recursively. Let $D^{(n)}_1 = [n]$.
\item If $D^{(n)}_i$ is already defined,  we pick the $\Ct^{(n),i}$-th card of $D^{(n)}_i$ (counted from the bottom of $D^{(n)}_i$) and let $\sigma_n(i)$ be the number written on this card. We remove this card from the old deck, i.e., we let $D^{(n)}_{i+1} = D^{(n)}_i \setminus \{\sigma_n(i)\}$, and we place it on the top of the new deck.
\item We continue this procedure until there are no cards left in the old deck.
\end{itemize}
\end{definition}



Definition \ref{def_gen_perm_model} is a generalization of  two important special cases: the Mallows model introduced in \cite{Mallows} and  the  $k$-card-minimum model (or  $k$CM for short) introduced in \cite{Travers}.

The following  definition of the Mallows model is proved to be equivalent to the original definition in \cite{GO, Mallows}.
\begin{definition} Let $n \in \mathbb{N}$ and $q \in \mathbb{R}_+\setminus \{1\}$.
    The $(n,q)$-Mallows model can be generated as in Definition \ref{def_gen_perm_model} if
        we sample $\Ct^{(n),i}$ according to the truncated geometric distribution
\begin{align}\label{distribution_Mallows}
        \mathbb{P}(\Ct^{(n),i} = l) = \frac{q^{l-1} - q^{l}}{1-q^{n-(i-1)}}, \qquad l=1,\dots, n-(i-1).
        \end{align}
\end{definition}
\begin{corollary}\label{cor_g_Mallows} For any $\beta \in \mathbb{R}$
    we recover the $(n,q)$-Mallows model with parameter $q=e^{-\frac{\beta}{n}}$ as a special case of  Definition \ref{def_gen_perm_model} if we choose
    \begin{equation}\label{eq_g_Mallows}
        \gt(x) = \begin{cases} \frac{e^{\beta \cdot (1-x)}-1}{e^{\beta}-1} & \text{ if } \; \; \beta \neq 0, \\
        1-x & \text{ if } \; \; \beta = 0.
        \end{cases}
    \end{equation}
\end{corollary}

In the  $k$CM model \cite{Travers}  the random variable $\Ct^{(n),i}$ from Definition \ref{def_gen_perm_model} is the minimum of $k$ i.i.d.\ random variables with uniform distribution on the set $\{ 1,\dots, n-(i-1) \}$.

\begin{corollary}\label{cor_g_kCM}
    The $k$CM \cite{Travers} arises as the special case of Definition \ref{def_gen_perm_model} if we choose
    \begin{align}\label{eq_g_kCM}
        \gt(x) = (1-x)^k.
    \end{align}
\end{corollary}

\begin{definition}[Permuton] A permuton $\mu$ is a probability measure on $[0,1]^2$ such that both marginal distributions are uniform on $[0,1]$.
    If $(X,Y) \sim \mu$ then $F_\mu$ denotes the joint cumulative distribution function of $(X,Y)$. Let us define the $\cup$-c.d.f.\ of $\mu$ by
        \begin{equation}\label{def_eq_cup_cdf}
        V_\mu(x,y)=x+y-F_\mu(x,y)=\mathbb{P}\left(\{X \leqslant x\} \cup \{ Y \leqslant y \}  \right).
    \end{equation}
\end{definition}

\begin{definition}[Permutations as permutons]\label{def_permutation_permuton} If $\sigma \in S_n$, let $\mu_\sigma$ denote the permuton  with joint p.d.f.\
$f_\sigma(x,y)=n \cdot \mathds{1}[ \sigma(\lceil n \cdot x \rceil)=\lceil n \cdot y \rceil ]$ if $(x,y)\in (0,1]^2$ and zero otherwise. Let $F_{\sigma}$ denote the joint c.d.f.\ of $\mu_\sigma$. Thus,
for $i,j \in \{0, 1, 2, \ldots, n\}$ we have
\begin{align}
\label{empirical_cdf}
F_\sigma\Big(\frac{i}{n},\frac{j}{n}\Big) &= \frac{1}{n}\sum_{t=1}^{i} \mathds{1}[\, \sigma(t) \leqslant j \, ]=\frac{1}{n} \cdot\left| [j] \setminus D^{(n)}_{i+1} \right| ,
\\ \label{cup_empirical_cdf}
V_\sigma\Big(\frac{i}{n},\frac{j}{n}\Big) &=
 \frac{1}{n}\sum_{t=1}^{n} \mathds{1}[\, t \leqslant i \, \text{ or } \, \sigma(t) \leqslant j \,] = \frac{i}{n} +\frac{j}{n}-F_\sigma\Big(\frac{i}{n},\frac{j}{n}\Big).
\end{align}
\end{definition}

In order to state the law of large numbers, we first define the limiting permuton.
\begin{definition}[Permuton arising from $\gt$]\label{def_permuton_from_gt}
Given $\gt$ as in Definition \ref{def_gen_perm_model}, let us  define the function $V_{\gt}: [0,1)\times [0,1) \to \mathbb{R}$ as follows. For each $y \in [0,1)$, let $V_{\gt}(\cdot,y)$ denote the unique solution of the initial value problem
\begin{align}\label{Vdef}
V_{\gt}(0,y) = y, \qquad \partial_x V_{\gt}(x,y) = \frac{\gt(V_{\gt}(x,y))}{\gt(x)}.
\end{align}
 Noting that \eqref{Vdef} is a separable ODE, for any $(x,y)\in[0,1)^2$ we obtain
    \begin{align}\label{V_in_terms_of_u}
        V_{\gt}(x,y) = u^{-1}(u(x) + u(y)), \quad
\text{where} \quad u(x) = \int_0^x \frac{1}{\gt(z)}\, \mathrm{d}z.
\end{align}
Our assumptions on $\gt$ (cf.\ Definition \ref{def_gen_perm_model}) imply $\lim_{x \to 1_-} u(x)=+\infty$, thus $V_{\gt}$ continuously extends to $[0,1]^2$ if we define $V_{\gt}(1,y)=V_{\gt}(x,1)=1$. Defining
\begin{equation}\label{def_eq_F_gt}
F_{\gt}(x,y):=x+y-V_{\gt}(x,y)\end{equation}
we obtain
$F_{\gt}(x,0) = 0$, $ F_{\gt}(0,y) = 0 $, $F_{\gt}(x,1) = x$,  $ F_{\gt}(1,y) = y $ and
\begin{equation}\label{joint_pdf_formula}
f_{\gt}(x,y):=  \partial_x \partial_y F_{\gt}(x,y)= \frac{-\gt'\left(V_{\gt}(x,y) \right) \gt\left( V_{\gt}(x,y) \right) }{\gt(x)\gt(y)} \geq 0, \qquad  x,y \in [0,1),
\end{equation}
thus one sees that $F_{\gt}$ is the c.d.f.\ (and therefore $V_{\gt}$ is the $\cup$-c.d.f.) of a permuton $\mu_{\gt}$.
\end{definition}

Our first main result states that the sequence of random permutations $\sigma_n$ from Definition \ref{def_gen_perm_model} converges in probability to the deterministic permuton $\mu_{\gt}$.

\begin{theorem}[Law of large numbers in permuton space]\label{permuton_convergence_thm} Given $\gt$ and $\sigma_n \sim \mathrm{PERM}(\gt,n)$ as in Definition \ref{def_gen_perm_model}, let $F_{\sigma_n}$ denote the (random) c.d.f.\ of the (random) permuton corresponding to $\sigma_n$ (cf.\ Definition \ref{def_permutation_permuton}). We have
    \begin{align}\label{cdf_convergence_eq}
        F_{\sigma_n}(x,y) \stackrel{\mathbb{P}}{\longrightarrow} F_{\gt}(x,y), \quad n \to \infty, \quad x,y \in [0,1].
    \end{align}
\end{theorem}


\begin{corollary}[Mallows and $k$CM limiting permutons]
$ $
\begin{itemize}
\item In the case of the Mallows model (cf.\ \eqref{eq_g_Mallows}) the formula \eqref{V_in_terms_of_u} becomes
\begin{equation}\label{mallows_permuton_limit}
    V_{\gt}(x,y) = \begin{cases} \frac{1}{\beta} \ln\left( \frac{e^{\beta(x+1)}+e^{\beta(y+1)}-e^{\beta(x+y)}-e^{\beta}  }{e^{\beta}-1} \right) & \text{ if } \; \; \beta \neq 0, \\
     x+y-xy  & \text{ if } \; \; \beta =0.
   \end{cases}
    \end{equation}
 \item    In the case of the $k$CM model (cf.\ \eqref{eq_g_kCM}) the formula \eqref{V_in_terms_of_u} becomes \begin{equation}\label{kcm_permuton_limit}
     V_{\gt}(x,y) = \begin{cases} 1 - ((1-x)^{1-k} + (1-y)^{1-k} -1)^{\frac{1}{1-k}} & \text{ if } \; \; k =2,3,\dots, \\
     x+y-xy  & \text{ if } \; \; k =1.
     \end{cases}
     \end{equation}
     \end{itemize}
\end{corollary}
Note that in the $\beta=0$ case of \eqref{mallows_permuton_limit} and in the $k=1$ case of \eqref{kcm_permuton_limit}  the function $V_{\gt}$ is the $\cup$-c.d.f.\ of the \emph{uniform permuton} (i.e., $\mu_{\gt}$ is the Lebesgue measure on $[0,1]^2$), corresponding to the fact that in both of these cases we have $\gt(x)=1-x$, and thus
the random permutation $\sigma_n \sim \mathrm{PERM}(\gt,n)$ generated according to
Definition \ref{def_gen_perm_model}
is actually uniformly distributed on $S_n$.

Note that the permuton limit  of the Mallows model has already been identified in \cite[Theorem 1.1]{starr} (before the invention of permutons)  and the p.d.f.\ $f_{\gt}(x,y)$ (cf. \eqref{joint_pdf_formula}) that we obtain from \eqref{mallows_permuton_limit} coincides with the p.d.f.\ $u(x,y)$ identified in \cite[Theorem 1.1]{starr}. In other words, we give an alternative proof of \cite[Theorem 1.1]{starr} using different methods. Our Theorem \ref{permuton_convergence_thm} is new in the case of the $k$CM model.

\begin{remark}[Symmetry] Note that if $\sigma_n \sim \mathrm{PERM}(\gt,n)$ then $\sigma^{-1}_n \sim \mathrm{PERM}(\gt,n)$
(cf.\ Proposition \ref{inverse_cor}). Putting this together with the identity
$F_{\sigma_n^{-1}}(x,y)\equiv F_{\sigma_n}(y,x)$ and the convergence result \eqref{cdf_convergence_eq} we see that $ F_{\gt}(x,y)\equiv F_{\gt}(y,x)$ must hold even without knowing the explicit formula for  $F_{\gt}$
(cf.\ \eqref{V_in_terms_of_u}, \eqref{def_eq_F_gt}).
\end{remark}

If $\tau \in S_k$ and $\sigma \in S_n$ and $k \leqslant n$ then we define the pattern density $t(\tau,\sigma)$ to be the fraction of $k$-tuples  $1 \leqslant i_1<\dots<i_k \leqslant n$ for which the relative order of $\sigma(i_1), \dots, \sigma(i_n)$  is given by the permutation $\tau$ (see \cite[Definition 1.1]{the_paper}). If $\mu$ is a permuton then we define $t(\tau,\mu)$ to be the probability of the event that the relative order of $Y_1,\dots,Y_k$ with respect to $X_1,\dots, X_k$ is given by $\tau$, where $(X_1,Y_1), \dots, (X_k,Y_k)$ are i.i.d.\ elements of $[0,1]^2$ with distribution $\mu$ (see \cite[Definitions 1.4, 1.5]{the_paper}).

\begin{corollary}[Law of large numbers for pattern densities]
\label{lln_pattern}
Theorem \ref{permuton_convergence_thm} and \cite[Lemma 5.3]{the_paper} together
imply that under the assumptions of Theorem \ref{permuton_convergence_thm} we have
\begin{equation}\label{lln_pattern_eq}
t(\tau, \sigma_n) \stackrel{\mathbb{P}}{\longrightarrow} t(\tau, \mu_{\gt}), \quad n \to \infty, \quad \tau \in S_k, \; k \in \mathbb{N}.
\end{equation}
\end{corollary}
In particular, if $\tau \in S_2$ is the permutation that swaps $1$ and $2$ then the density $t(\tau,\mu_{\gt})$ of inversions in $\mu_{\gt}$ can be shown to be equal to  $ 2 \int_{0}^1 u(s) \gt(s) \text{ds}$ (where $u$ is defined in \eqref{V_in_terms_of_u}), and if we choose $\gt(x) = (1-x)^k$ (cf.\ \eqref{eq_g_kCM}) then
$ 2 \int_{0}^1 u(s) \gt(s) \text{ds}= \frac{1}{k+1} $ and thus
Corollary \ref{lln_pattern} recovers \cite[Theorem 2.2]{Travers}, i.e., the weak law of large numbers for the number of inversions
 in the $k$CM model. Note that the method of proof of \cite[Theorem 2.2]{Travers} does not seem to generalize to the case of \eqref{lln_pattern_eq} when $\tau \in S_k$,  $k \geq 3$.
 Also note that the special case of \eqref{lln_pattern_eq} for the Mallows model is stated in \cite[Section 1.2.2]{Dubach}.


If we denote by $V_\beta(x,y)$ the $\cup$-c.d.f.\ that appears in  \eqref{mallows_permuton_limit} and by $\mu_\beta$ the corresponding permuton then $\lim_{\beta \to \infty} V_\beta(x,y)=\max\{x,y\}$, thus the weak limit of the probability measures $\mu_\beta$ as $\beta \to \infty$ is the \emph{identity permuton}, i.e., the law of the distribution of $(X,X)$, where $X \sim \mathrm{UNI}[0,1]$.
The same thing happens if we let $k \to \infty$ in \eqref{kcm_permuton_limit}. Our next result (Theorem \ref{logistic_limit_thm}) is a more precise description of the  concentration of these measures around the diagonal.

\begin{assumption}[Convergence of scaled logarithmic derivative]\label{assumption_kk} For each $\kk \geq 1$, let $\gt_{\kk}$ denote a function that satisfies the assumptions for $\gt$ in Definition \ref{def_gen_perm_model}.
Let $\gamma: [1,+\infty) \to \mathbb{R}_+$ denote a scaling function.
For each $\kk \geq 1$, let
       \begin{align}\label{def_psi_k}
        \psi_\kk(s) :=  - \frac{1}{\gamma(\kk)} \cdot \frac{\gt_{\kk}'(s)}{\gt_{\kk}(s)} , \qquad s \in [0,1).
    \end{align}
    Let $\psi: [0,1) \to \mathbb{R}$. We assume  that
\begin{enumerate}[(i)]
\item\label{ass_gammak_to_infty} $\lim_{\kk \to \infty} \gamma(\kk)=+\infty$,
 \item\label{ass_psik_unif_equico}  $\psi_{\kk}(\cdot),\; \kk \geq 1$ are uniformly equicontinuous on compact subintervals of $[0,1)$,
\item \label{ass_gammak_unif_conv} $\psi_{\kk} \longrightarrow \psi$ uniformly on each compact subinterval of $[0,1)$ as $\kk \to \infty$,

     \item \label{ass_psi_pos} $\psi(s)>0$ for all $s \in [0,1)$.
    \end{enumerate}
     Note that it follows from the above assumptions that the limit $\psi$ is continuous on $[0,1)$.
\end{assumption}
\begin{remark}[Scaling of logarithmic derivatives of Mallows and $k$CM models]\label{remark_kcm_mallows_satisfy_assump} $ $
\begin{itemize}
\item If
$   \gt_\kk(s) = \frac{e^{\kk\cdot (1-s)}-1}{e^{\kk}-1}$ (cf.\ \eqref{eq_g_Mallows}) then
the conditions of Assumption \ref{assumption_kk} are satisfied if we choose $\gamma(\kk)=\kk$ and in this case we have $\psi(s)= 1$ for all $s \in [0,1)$.
\item
  If  $\gt_\kk(s)=(1-s)^{\kk}$ (cf.\ \eqref{eq_g_kCM}) then
the conditions of Assumption \ref{assumption_kk} are satisfied if we choose $\gamma(\kk)=\kk$ and in this case we have $\psi(s)= \frac{1}{1-s}$ for all $s \in [0,1)$.
\end{itemize}
\end{remark}

\begin{theorem}[Logistic limit theorem]
\label{logistic_limit_thm}
 Let $\gt_{\kk}, \kk \geq 1$
satisfy Assumption \ref{assumption_kk}. Let
us define the random variables
\begin{equation}\label{def_eq_UkVk}
\mathcal{U}_\kk := \frac{X_\kk + Y_\kk}{2}, \qquad \mathcal{V}_\kk := \gamma(\kk)\cdot \frac{X_\kk-Y_\kk}{2}, \end{equation}
where the random vector $(X_\kk, Y_\kk)$ has  distribution $\mu_{\gt_{\kk}}$ (cf.\ Definition \ref{def_permuton_from_gt}). We have
\begin{align}\label{eq_limit_thm_logistic}
    (\mathcal{U}_\kk, \mathcal{V}_\kk) \Longrightarrow \big(U, \frac{1}{\psi(U)}\cdot L\big), \qquad \kk \to \infty,
\end{align}
where $U \sim \mathrm{UNI}[0,1]$,  $L$ is independent of $U$ and $L$ has p.d.f. $\frac{2}{(e^t + e^{-t})^2}, \, t \in \mathbb{R}$.
\end{theorem}
Note that the random variable $L$ has \emph{logistic distribution} and the c.d.f.\ of $L$ is the logistic curve: $\mathbb{P}(L \leqslant x)=\frac{1}{1+e^{-2x}}, \, x \in \mathbb{R}$. In plain words,
Theorem \ref{logistic_limit_thm} says that if $\kk$ is large then the mass of the distribution of $(X_\kk, Y_\kk)$ is mostly concentrated in a thin tube of width proportional to $\frac{1}{\gamma(\kk)}$ around the diagonal (and therefore the distribution of $\mathcal{U}_\kk$ is close to $\mathrm{UNI}[0,1]$),  and if we condition
on $\mathcal{U}_\kk=s$ (where $s \in [0,1)$) then the conditional distribution of $\psi(s) \cdot \mathcal{V}_\kk$ is close to that of $L$.
The result of Theorem \ref{logistic_limit_thm} (together with Remark \ref{remark_kcm_mallows_satisfy_assump}) confirms the conjecture of Travers formulated in \cite[Section 6]{Travers} about the similarity of the band structure of the
$k$CM and the $(n,q)$-Mallows models around the diagonal
 when $q \approx 1-\frac{k}{n}$ and $k$ is large,  see Figure \ref{fig:band} for an illustration.

\begin{figure}[t]
\centering
\includegraphics[width=13cm]{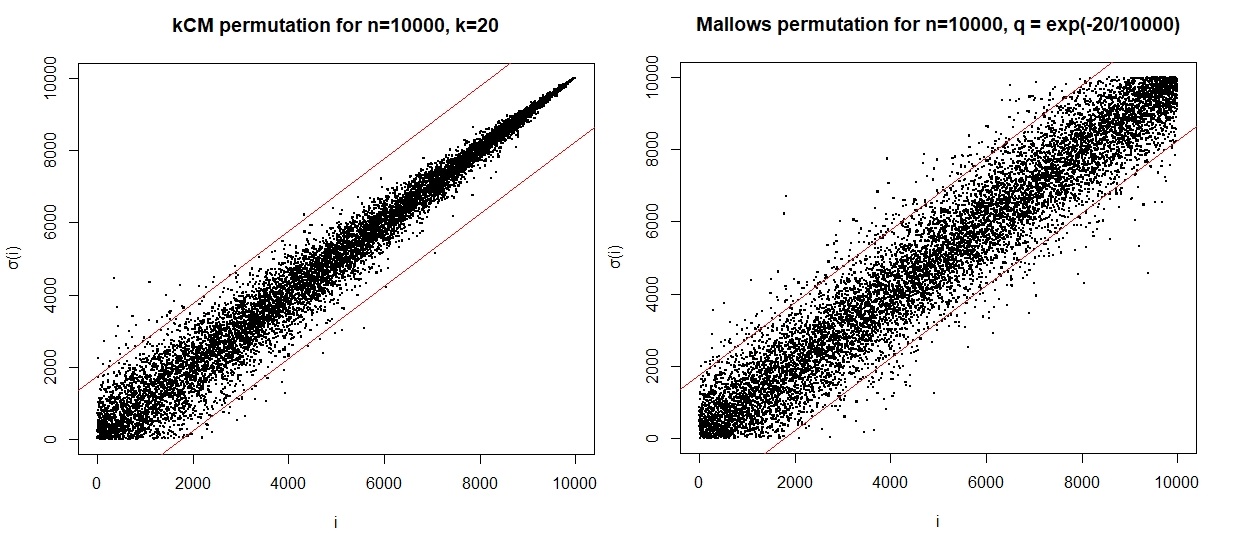}\vspace{-0.5em}
\caption{The image on the left is the matrix of a random permutation generated according to the $k$CM with parameters $n=10^4$ and $k=20$. The image on the right is the matrix of a random permutation generated according to the $(n,q)$-Mallows  model with parameters $n=10^4$ and $q=e^{\frac{-\beta}{n}}$, where $\beta=20$. Note the visible difference: $\psi(s)=\frac{1}{1-s}$ on the left, while $\psi(s)\equiv 1$ on the right, cf.\ Remark \ref{remark_kcm_mallows_satisfy_assump}.}
\label{fig:band}
\end{figure}

Note that the logistic distribution also appeared in \cite[Proposition 6.1]{bclogistic} as the scaling limit  of the displacement of particles in the context of Mallows product measures (a certain family of $q$-exchangeable probability measures on bijections $\mathbb{Z} \to \mathbb{Z}$) as $q \to 1$.

\begin{remark}[CLT]
In an upcoming paper we plan to prove that
$ \sqrt{n} \cdot \left( F_{\sigma_n}- F_{\gt} \right)$ (cf.\ \eqref{cdf_convergence_eq}) weakly converges to a centered Gaussian process defined on $[0,1]^2$ as $n \to \infty$ (generalizing the recent result of \cite{Bufetov_clt} pertaining to the Mallows model) and to derive that $\sqrt{n} \cdot \left(t(\tau, \sigma_n) - t(\tau, \mu_{\gt})\right) \Rightarrow \mathcal{N}(0,\sigma^2)$ (cf.\ \eqref{lln_pattern_eq}) for some $\sigma=\sigma(\gt,\tau) \in \mathbb{R}_+$  as a corollary (generalizing \cite[Theorem 2.3]{Travers}, i.e., the CLT for the number of inversions in the $k$CM). Note that our Proposition \ref{prop_tight} can be viewed as a first step in this direction.
\end{remark}

\section{Law of large numbers in permuton space}
In this section we prove  Theorem \ref{permuton_convergence_thm}. The key observation is the following lemma.
\begin{definition}[Natural filtration]\label{filtration}
Let $\sigma_n \in S_n$ be a $\gt$-random permutation.
    For $i=0,1, 2, \ldots n$ 
    let
    $\mathcal{F}^n_{i}$  denote the $\sigma$-field generated by $\sigma_n(1), \dots, \sigma_n(i)$.
\end{definition}
Recalling Definition \ref{def_permutation_permuton} we see that for any $j \in [n]$ the random variables $F_{\sigma_n}\big(\frac{i}{n},\frac{j}{n}\big)$
and $V_{\sigma_n}\big(\frac{i}{n},\frac{j}{n}\big)$ are both
$\mathcal{F}^n_{i}$-measurable.

A variant of the following lemma in the case of the Mallows model is \cite[Lemma 3.16]{GP}.
\begin{lemma}[Markov property]
  Let $\sigma_n \in S_n$ be a $\gt$-random permutation. Then for any $j \in [n]$ the sequence of random variables $V_{\sigma_n}\big(\frac{1}{n},\frac{j}{n}\big), V_{\sigma_n}\big(\frac{2}{n},\frac{j}{n}\big), \dots, V_{\sigma_n}\big(\frac{n}{n},\frac{j}{n}\big) $ (cf.\ Definition \ref{def_permutation_permuton}) forms a Markov chain. Specifically, for any $i,j \in [n]$ we have
   \begin{align}\label{cup_expectation}
  \mathbb{E}\Big(V_{\sigma_n}\Big(\frac{i}{n}, \frac{j}{n}\Big) - V_{\sigma_n}\Big(\frac{i-1}{n},\frac{j}{n}\Big)\: \big| \: \mathcal{F}_{i-1}^{n}\Big) = \frac{1}{n} \cdot \frac{\gt\Big(V_{\sigma_n}(\frac{i-1}{n}, \frac{j}{n})\Big)}{\gt(\frac{i-1}{n})}.
   \end{align}
\end{lemma}
\begin{proof}
Recalling \eqref{cup_empirical_cdf} we
observe that
$n \cdot V_{\sigma_n}\Big(\frac{i}{n}, \frac{j}{n}\Big) - n \cdot V_{\sigma_n}\Big(\frac{i-1}
{n},\frac{j}{n}\Big)$ can only take two possible values: it is $1$ if
$\sigma_n(i)>j$ and $0$ otherwise. Therefore, by Definition \ref{def_gen_perm_model} we have
\begin{multline*}
n \cdot\mathbb{E}\Big( V_{\sigma_n}\Big(\frac{i}{n}, \frac{j}{n}\Big) - V_{\sigma_n}\Big(\frac{i-1}{n},\frac{j}{n}\Big)\: \big| \: \mathcal{F}_{i-1}^{n}\Big) =
      \mathbb{P}\big( \sigma_n(i)>j \, \big| \, \mathcal{F}_{i-1}^{n}\big)= \\
       \mathbb{P}\big( \Ct^{(n), i} > |D^{(n)}_i \cap [j]|  \, \big| \, \mathcal{F}_{i-1}^{n}\big)
      \stackrel{\eqref{empirical_cdf}}{=}
     \mathbb{P}\big(\Ct^{(n), i}  > j - nF_{\sigma_n}\left(\frac{i-1}{n},\frac{j}{n}\right) \, | \, \mathcal{F}_{i-1}^{n}  \big)
    \stackrel{(\bullet)}{=} \\  \frac{\gt\big(\frac{i-1}{n} + \frac{j}{n} - F_{\sigma_n}(\frac{i-1}{n}, \frac{j}{n})\big)}{\gt(\frac{i-1}{n})} \stackrel{\eqref{cup_empirical_cdf}}{=}  \frac{\gt(V_{\sigma_n}(\frac{i-1}{n}, \frac{j}{n}))}{  \gt(\frac{i-1}{n})},
\end{multline*}
where in $(\bullet)$ we used that
$F_{\sigma_n}\big(\frac{i-1}{n},\frac{j}{n}\big)$ is
$\mathcal{F}^n_{i-1}$-measurable,
$\Ct^{(n), i}$ is independent of
$\mathcal{F}_{i-1}^{n}$ and that
$\mathbb{P}(\Ct^{(n), i} \geq l)=\gt\big(\frac{(i-1)+(l-1)}{n}\big) /\gt\big( \frac{i-1}{n} \big)$  for any $l=1,\dots, n-(i-1)$
by \eqref{nu_prob_def_eq}.
\end{proof}

In order to relate \eqref{cup_expectation} to \eqref{Vdef}, we need some technical estimates.

Recall from Definition \ref{def_gen_perm_model} that we assumed that $\gt$ is continuously differentiable on $[0,1]$, thus $\Vert \gt'  \Vert_\infty= \max_{0 \leqslant x \leqslant 1} |\gt'(x)|<+\infty$. Recall the notion $V_{\gt}$  from Definition \ref{def_permuton_from_gt}.
\begin{claim}\label{claim_Lipschitz} For any $\varepsilon \in (0,1)$
and for all $x \in [0, 1-\varepsilon]$ and all $y \in [0,1]$ we have
\begin{align}\label{lip_const}
    \Big| \frac{\mathrm{d}}{\mathrm{d}x} \frac{\gt(V_{\gt}(x,y))}{\gt(x)}\Big| \leqslant \frac{ 2 \Vert \gt'  \Vert_\infty}{(\gt(1-\varepsilon))^2} =: L(\varepsilon)<+\infty.
\end{align}
\end{claim}
We omit the details of the proof of Claim \ref{claim_Lipschitz} as it follows  (\ref{Vdef})  and the fact that $\gt$ is a strictly decreasing function on $[0,1]$
that satisfies $\gt(0)=1$ and $\gt(1)=0$.


\begin{claim}\label{claim_int_approx}
   For every $\varepsilon \in (0,1)$ and  $i, j \in \{1, 2, \ldots, n\}$ such that $\frac{i}{n} \leqslant 1- \varepsilon$, we have
    \begin{align}\label{eq_L_eps_bound}
        \Big|\int_{\frac{i-1}{n}}^{\frac{i}{n}}\frac{\gt\big(V_{\gt}(x,\frac{j}{n})\big)}{\gt(x)} \mathrm{dx} - \frac{1}{n}\cdot \frac{\gt\big(V_{\gt}(\frac{i-1}{n}, \frac{j}{n})\big)}{\gt(\frac{i-1}{n})}\Big| \leqslant \frac{1}{n^2}
\frac{L(\varepsilon)}{2}.
    \end{align}
\end{claim}
We omit the proof of Claim \ref{claim_int_approx} as it easily follows from Claim \ref{claim_Lipschitz}.

\begin{definition}[Scaled fluctuation]\label{def_Z}
    For fixed $i,j \in \{0,1, 2, \ldots, n\}$ we denote
    \begin{align}\label{def_eq_Z}
        Z_{i,j}^n := \sqrt{n} \cdot  \Big(V_{\sigma_n}\big(\frac{i}{n}, \frac{j}{n}\big) - V_{\gt}\big(\frac{i}{n}, \frac{j}{n}\big)\Big).
    \end{align}
\end{definition}
\begin{proposition}[Tightness]\label{prop_tight}
For any $\varepsilon \in (0,1)$ and $K \in \mathbb{R}_+$ we have
\begin{equation}\label{Z_deviation_prob_bound}
  \max_{1 \leqslant i \leqslant \lfloor (1-\varepsilon)n \rfloor} \; \max_{j \in [n]} \; \mathbb{P}\left( |Z_{i,j}^n| \geqslant (K+1)\cdot L(\varepsilon) \cdot e^{L(\varepsilon)}+K \right) \leqslant \frac{1}{K^2},
\end{equation}
\end{proposition}
\begin{proof} Let us fix  $j \in [n]$ for the rest of this proof. Note that $V_{\sigma_n}\big(\frac{0}{n}, \frac{j}{n}\big) = V_{\gt}\big(\frac{0}{n}, \frac{j}{n}\big)=\frac{j}{n}$, thus
$Z_{0,j}^n = 0$.
   Due to the Doob decomposition theorem, we can write $Z_{i,j}^n = M_i^n + A_i^n$, where $(M_i^n)_{i=0}^n$ is a martingale satisfying $M_0^{n} = 0$ and $(A_i^n)_{i=1}^{n}$ is a predictable process with respect to the filtration $(\mathcal{F}_i^n)_{i=0}^{n}$ introduced in Definition \ref{filtration}. In particular, we have
    \begin{align}
    M_0^n=0, \quad
        M_i^n - M_{i-1}^{n} &=  (Z_{i,j}^n - Z_{i-1,j}^{n}) - \mathbb{E}\big(Z_{i,j}^n - Z_{i-1,j}^n \:\big|\: \mathcal{F}_{i-1}^{n}\big), &  i \in [n], \\
        \label{predictable_increment}
        A_0^n=Z^n_{0,j}=0, \quad
        A_i^{n} -A_{i-1}^{n} &=   \mathbb{E}\big(Z_{i,j}^n - Z_{i-1,j}^{n}\: \big| \mathcal{F}_{i-1}^{n}\big),  & i \in [n].
    \end{align}
     First we bound the martingale part. Since $\mathbb{E}\big(M_{\lfloor (1-\varepsilon)\cdot n \rfloor}^n\big)=\mathbb{E}(M^n_0)=0$, we have
    \begin{align*}
        \mathbb{E}\big((M_{\lfloor (1-\varepsilon)\cdot n \rfloor}^n)^2\big)
        = \mathrm{Var}\big(M^n_{\lfloor (1-\varepsilon)n \rfloor}\big)
        = \sum_{i=1}^{\lfloor (1-\varepsilon)\cdot n \rfloor} \mathrm{Var} (M_i^n - M_{i-1}^{n}) \stackrel{(*)}{\leqslant} \lfloor (1-\varepsilon)\cdot n \rfloor \cdot \frac{1}{n} \leqslant 1,
    \end{align*}
    where in $(*)$ we used
$\mathrm{Var} (M_i^n - M_{i-1}^{n} \, | \, \mathcal{F}_{i-1}^{n} )=n \cdot \mathrm{Var} \left( V_{\sigma_n}\big(\frac{i}{n}, \frac{j}{n}\big)- V_{\sigma_n}\big(\frac{i-1}{n}, \frac{j}{n}\big) \, | \, \mathcal{F}_{i-1}^{n} \right) \leqslant \frac{n}{n^2}$ as well as the identity
$\mathrm{Var} (M_i^n - M_{i-1}^{n})=\mathbb{E}\left( \mathrm{Var} (M_i^n - M_{i-1}^{n} \, | \, \mathcal{F}_{i-1}^{n} ) \right) $, which follows from the law of total variance and the martingale property.

    Given some $K \in \mathbb{R}_+$ let us introduce the good event
\begin{equation}
    B_{n,K} := \left\{ \max_{1 \leqslant i \leqslant \lfloor (1-\varepsilon)n\rfloor} |M_i^n| < K \right\}.
\end{equation}
 By Doob's inequality we obtain the bound $
       \mathbb{P}(B_{n,K}^c) \leqslant \frac{\mathbb{E}((M^{n}_{\lfloor (1-\varepsilon)\cdot n \rfloor})^2)}{K^2} \leqslant \frac{1}{K^2}$.

     Now let us bound the predictable process $(A_i^n)$. We have
    \begin{multline}\label{A_increment_formula}
       A_i^n - A_{i-1}^n
       \stackrel{\eqref{predictable_increment}}{=}  \mathbb{E}(Z_{i,j}^n - Z_{i-1,j}^n \: \big| \: \mathcal{F}_{i-1}^{n})
        \stackrel{\eqref{def_eq_Z}}{=} \\ \sqrt{n}\cdot \mathbb{E}\Big(V_{\sigma_n}\Big(\frac{i}{n}, \frac{j}{n}\Big) - V_{\sigma_n}\Big(\frac{i-1}{n}, \frac{j}{n}\Big) \: \big| \: \mathcal{F}_{i-1}^n \Big) - \sqrt{n} \cdot \Big(V_{\gt}\Big(\frac{i}{n}, \frac{j}{n}\Big) - V_{\gt}\Big(\frac{i-1}{n}, \frac{j}{n}\Big)\Big)\\
    \stackrel{\eqref{cup_expectation}, \eqref{Vdef}}{=}
         \sqrt{n} \cdot \Big( \frac{1}{n} \cdot \frac{\gt(V_{\sigma_n}(\frac{i-1}{n}, \frac{j}{n}))}{\gt(\frac{i-1}{n})} - \int_{\frac{i-1}{n}}^{\frac{i}{n}} \frac{\gt(V_{\gt}(x,\frac{j}{n}))}{\gt(x)} \mathrm{dx}\Big),
\end{multline}
thus for any $i$ satisfying  $1 \leqslant i \leqslant \lfloor (1-\varepsilon)\cdot n \rfloor$ we can bound
\begin{multline}\label{A_diff_boundd}
| A_i^n - A_{i-1}^n| \stackrel{\eqref{A_increment_formula}, \eqref{eq_L_eps_bound}}{\leqslant}
                   \frac{1}{\sqrt{n}} \cdot \Big| \frac{\gt(V_{\sigma_n}(\frac{i-1}{n}, \frac{j}{n}))}{\gt(\frac{i-1}{n})} - \frac{\gt(V_{\gt}(\frac{i-1}{n}, \frac{j}{n}))}{\gt(\frac{i-1}{n})}\Big| +  \frac{L(\varepsilon)}{2 n^{3/2} } \stackrel{(*)}{\leqslant} \\
 \frac{1}{\sqrt{n}} \cdot \frac{ \Vert \gt' \Vert_{\infty} \cdot \left| V_{\sigma_n}(\frac{i-1}{n}, \frac{j}{n}) - V_{\gt}(\frac{i-1}{n}, \frac{j}{n})\right|}{\gt(1-\varepsilon)} + \frac{L(\varepsilon)}{2 n^{3/2} }  \stackrel{\eqref{def_eq_Z}}{=}
\frac{1}{n} \cdot \frac{\Vert \gt' \Vert_{\infty}}{\gt(1-\varepsilon)} \cdot |Z_{i-1,j}^n| +  \frac{L(\varepsilon)}{2 n^{3/2} } \\ \leqslant
 \frac{L(\varepsilon)}{n}  \cdot |Z_{i-1,j}^n| + \frac{L(\varepsilon)}{2 n^{3/2} }=
  \frac{ L(\varepsilon)}{n}  \cdot |M_{i-1}^n+A_{i-1}^n| + \frac{L(\varepsilon)}{2 n^{3/2} }.
\end{multline}
where in $(*)$ we used that $0 \leqslant \frac{i-1}{n} \leqslant 1-\varepsilon$ and that $\gt$ is decreasing.

Recalling that $A_0^n=0$,  for any
$1 \leqslant l \leqslant \lfloor (1-\varepsilon)\cdot n \rfloor$ we obtain
\begin{equation}
   |A_l^n| \leqslant  \sum_{i=1}^{l} | A_i^n - A_{i-1}^n| \stackrel{\eqref{A_diff_boundd}}{\leqslant}
      \frac{L(\varepsilon)}{2\sqrt{n}} + \sum_{i=1}^{l} \frac{L(\varepsilon)}{n} \cdot |M_{i-1}^n+A_{i-1}^n|.
\end{equation}
Thus, if the event $B_{n,K}$ occurs then $ |A_l^n|
\leqslant (K+1) \cdot L(\varepsilon) + \frac{L(\varepsilon)}{n} \sum_{i=1}^l |A_{i-1}^n|$ holds for any $1 \leqslant l \leqslant \lfloor (1-\varepsilon)\cdot n \rfloor$.
By induction on $l$ one obtains that if $B_{n,K}$ occurs then
for all $1 \leqslant l \leqslant \lfloor (1-\varepsilon)\cdot n \rfloor$ we have
$  |A_l^n| \leqslant (K+1) \cdot L(\varepsilon) \cdot \Big(1+\frac{L(\varepsilon)}{n}\Big)^{l-1}$.
 Since $Z_{i,j}^n = M_i^n + A_i^n$, this bound on $ |A_l^n|$ and
 $ \mathbb{P}(B_{n,K}^c) \leqslant \frac{1}{K^2}$  together imply
 $ \mathbb{P}\left(\max_{1 \leqslant i \leqslant \lfloor (1-\varepsilon)n \rfloor} |Z_i^n| \geqslant (K+1)\cdot L(\varepsilon) \cdot e^{L(\varepsilon)}+K \right) \leqslant \frac{1}{K^2} $,
 from which
 the desired \eqref{Z_deviation_prob_bound} follows.
\end{proof}

\begin{proof}[Proof of Theorem \ref{permuton_convergence_thm}]
First note that \eqref{cdf_convergence_eq} trivially holds if $x=1$, since
$F_{\sigma_n}(1,y)=y=  F_{\gt}(1,y)$ holds for any $n \in \mathbb{N}$ and any $y \in [0,1]$. Similarly, \eqref{cdf_convergence_eq} trivially holds if $x=0$, $y=0$ or $y=1$. Let us thus fix $x,y \in (0,1)$.
Let us fix $\varepsilon \in (0,1)$ such that $x \leqslant 1-\varepsilon$.
Let us define $i_n:= \lceil n \cdot x \rceil $ and $j_n:=\lceil n \cdot y \rceil$.
Noting that $|F_\mu(x,y)-F_\mu(\frac{i_n}{n}, \frac{j_n}{n})| \leqslant \frac{2}{n}$ holds for the joint c.d.f.\ $F_\mu$ of any permuton (since the marginals have uniform distribution) and recalling the identity \eqref{def_eq_cup_cdf} we see that it
is enough to show
 $|V_{\sigma_n}(\frac{i_n}{n}, \frac{j_n}{n}) - V_{\gt}(\frac{i_n}{n}, \frac{j_n}{n})|\stackrel{\mathbb{P}}{\longrightarrow}0$ as  $n \to \infty$ in order to prove
    \eqref{cdf_convergence_eq}. This convergence result follows from Definition \ref{def_Z} and  Proposition \ref{prop_tight}.
\end{proof}


\section{Band structure of the limiting permuton}
In this section we prove Theorem \ref{logistic_limit_thm}. Throughout this section we assume that the functions $\gt_{\kk}, \kk \geq 1$
satisfy Assumption \ref{assumption_kk}.
Let us denote by $h_{\kk}$ the joint p.d.f.\ of $(\mathcal{U}_\kk, \mathcal{V}_\kk)$ (cf.\ \eqref{def_eq_UkVk}). By Scheff\'e's lemma it is enough to prove the pointwise convergence
\begin{equation}\label{limit_thm_pdf}
\lim_{\kk \to \infty} h_{\kk}(s,t)=  \frac{2 \cdot \psi(s) }{(e^{\psi(s)\cdot t}+e^{-\psi(s)\cdot t})^2}, \quad s \in (0,1), \; t \in \mathbb{R}
\end{equation}
in order to conclude \eqref{eq_limit_thm_logistic}, since
the r.h.s.\ of \eqref{limit_thm_pdf} is the joint p.d.f.\ of $\big(U, \frac{1}{\psi(U)}\cdot L\big)$.

We have
    $h_{\kk}(s,t) := \frac{2}{\gamma(\kk)} \cdot f_{\gt_{\kk}}\left(s+\frac{t}{\gamma(\kk)}, s - \frac{t}{\gamma(\kk)}\right) $
    for any $s \in (0,1)$ and $t \in \mathbb{R}$, where $f_{\gt_{\kk}}$ is defined in \eqref{joint_pdf_formula}.
Thus we obtain
    \begin{align}\label{u_density}
        h_\kk(s,t) \stackrel{ \eqref{joint_pdf_formula} }{=} -\frac{1}{\gamma(\kk)} \cdot \frac{((\gt_\kk)^2)'\left( V_\kk( s+\frac{t}{\gamma(\kk)}, s-\frac{t}{\gamma(\kk)} ) \right)}{\gt_\kk(s+\frac{t}{\gamma(\kk)})\cdot \gt_\kk(s-\frac{t}{\gamma(\kk)})},
    \end{align}
where $V_\kk(x,y):=u_\kk^{-1}(u_\kk(x) + u_\kk(y))$ and  $u_\kk(x) = \int_0^x \frac{1}{\gt_\kk(z)}\, \mathrm{d}z$, cf.\
\eqref{V_in_terms_of_u}.

Note that from \eqref{def_psi_k} and $\gt_\kk(0)=1$ we obtain
\begin{equation}\label{gt_from_psi_k}
\gt_\kk(s) = \exp\left(-\gamma(\kk)\int_0^s \psi_\kk(u) \mathrm{du}\right), \qquad s \in [0,1).
\end{equation}
It follows from  Assumption \ref{assumption_kk}
that for all $\varepsilon \in (0,1)$ we have
\begin{equation}\label{min_kk_eps}
   \lim_{\kk \to \infty}  m(\kk, \varepsilon)=m(\varepsilon)>0, \; \text{where}\; m(\kk, \varepsilon):=  \min_{0 \leqslant x \leqslant 1-\varepsilon} \psi_\kk(x), \quad  m(\varepsilon):=\min_{0 \leqslant x \leqslant 1-\varepsilon} \psi(x).
\end{equation}

\begin{lemma}[Estimating $\gt_\kk$] For any $\varepsilon \in (0,1)$ and $T \in \mathbb{R}_+$ we have
    \begin{align}\label{psi_exp}
   \lim_{\kk \to \infty}
   \delta(\kk, \varepsilon,T)=0, \quad \text{where} \quad
  \delta(\kk, \varepsilon,T):=    \sup_{x \in [\varepsilon, 1-\varepsilon]} \sup_{y \in [-T,T]}
\left|
\frac{\gt_\kk(x) e^{-\psi(x)\cdot y}}{\gt_\kk\big(x+\frac{y}{\gamma(\kk)}\big)}
 - 1 \right| =0.
    \end{align}
\end{lemma}
\begin{proof} Let us fix $\varepsilon$ and $T$.
If $\kk$ is big enough then
$x \in [\varepsilon, 1-\varepsilon]$, $y \in [-T,T]$ and Assumption \ref{assumption_kk}\eqref{ass_gammak_to_infty}
together imply $x+\frac{y}{\gamma(\kk)} \in [0, 1-\frac{\varepsilon}{2}]$. Let us assume that this holds. We have
    \begin{align*}
        \frac{\gt_\kk(x) e^{-\psi(x)\cdot y}}{\gt_\kk\big(x+\frac{y}{\gamma(\kk)}\big)}   \stackrel{\eqref{gt_from_psi_k}}{=} \exp\Bigg(\gamma(\kk) \int_x^{x+\frac{y}{\gamma(\kk)}}\big(\psi_\kk(u) - \psi(x)\big)\mathrm{d}u\Bigg).
    \end{align*}
    In order to prove \eqref{psi_exp}, we observe that
\begin{multline}\label{sup_sup_int_bound}
\sup_{x \in [\varepsilon, 1-\varepsilon]} \sup_{y \in [-T,T]}  \left|  \gamma(\kk) \int_x^{x+\frac{y}{\gamma(\kk)}}\big(\psi_\kk(u) - \psi(x)\big)\mathrm{d}u \right| \leqslant \\ T \cdot \sup_{x \in [\varepsilon, 1-\varepsilon]} \sup_{ \frac{-T}{\gamma(\kk)}\leqslant z \leqslant \frac{T}{\gamma(\kk)}}|\psi_\kk(x+z) - \psi(x)|
\end{multline}
holds, and  it follows from parts \eqref{ass_psik_unif_equico} and \eqref{ass_gammak_unif_conv} of Assumption \ref{assumption_kk} that the r.h.s.\ of \eqref{sup_sup_int_bound} goes to zero as $\kk \to \infty$. The proof of \eqref{psi_exp} is complete.
\end{proof}

\begin{lemma}[Estimating $u_\kk$]\label{lemma_u_k_x}
For any $\varepsilon \in (0,1)$ we have
    \begin{align}\label{eq_u_kk_x}
\limsup_{\kk \to \infty} \sup_{x \in [\varepsilon, 1-\varepsilon]}
\left| u_\kk(x) \cdot \gamma(\kk) \cdot \gt_{\kk}(x) - \frac{1}{\psi(x)}   \right|=0.
    \end{align}
    \end{lemma}
    \begin{proof} Let us fix $\varepsilon \in (0,1)$. For any  $T \in \mathbb{R}_+$ and for large enough $\kk$ we have
\begin{multline}\label{u_split_T_ints}
     u_\kk(x) \cdot \gamma(\kk) \cdot \gt_{\kk}(x)\stackrel{\eqref{V_in_terms_of_u} }{=} \gamma(\kk) \int_0^x \frac{\gt_\kk(x)}{\gt_\kk(z)} \, \mathrm{d}z=
     \int_0^{\gamma(\kk)x} \frac{\gt_\kk(x)}{\gt_\kk(x-\frac{y}{\gamma(\kk)})}\, \mathrm{d}y
     \stackrel{ \eqref{gt_from_psi_k}}{=}
     \\ \int_0^T \frac{\gt_\kk(x)}{\gt_\kk(x-\frac{y}{\gamma(\kk)})}\, \mathrm{d}y + \int_T^{\gamma(\kk)x} \exp\left( -\gamma(\kk)\int_{x-\frac{y}{\gamma(\kk)}}^x \psi_\kk(u) \mathrm{du} \right)\, \mathrm{d}y.
\end{multline}
We have $\left|\int_0^T \frac{\gt_\kk(x)}{\gt_\kk(x-\frac{y}{\gamma(\kk)})}\, \mathrm{d}y -\int_0^T e^{-\psi(x) \cdot y} \, \mathrm{d}y \right| \leqslant T \cdot \delta(\kk,\varepsilon,T)$ (cf.\ \eqref{psi_exp}), while the second summand on the r.h.s.\ of
\eqref{u_split_T_ints} is  bounded by $\int_T^{\gamma(\kk)x} e^{ -m(\kk,\varepsilon) \cdot y }\, \mathrm{d}y
\leqslant \frac{e^{-m(\kk,\varepsilon)\cdot T}}{m(\kk, \varepsilon)}$ (cf.\ \eqref{min_kk_eps}). Noting that $\left| \int_0^T e^{-\psi(x) \cdot y}\, \mathrm{d}y -\frac{1}{\psi(x)} \right| =
\frac{e^{-\psi(x) \cdot T}}{\psi(x)} \leqslant \frac{e^{-m(\varepsilon) \cdot T}}{m(\varepsilon)}$, we obtain from
  \eqref{min_kk_eps} and \eqref{psi_exp} that the $\limsup$ that appears on the l.h.s.\ of \eqref{eq_u_kk_x} is at most $2 \cdot \frac{e^{-m(\varepsilon) \cdot T}}{m(\varepsilon)}$. Since the choice of $T$ was arbitrary and $m(\varepsilon)>0$, the proof of Lemma \ref{lemma_u_k_x} is complete.
    \end{proof}


\begin{corollary}\label{cor_uu_frac_lim}
    For any $s \in (0,1)$ and $y \in \mathbb{R}$ we have
    \begin{multline}\label{u_fraction_lim}
        \lim_{\kk \to \infty} \frac{u_{\kk}\left( s+ \frac{y}{\gamma(\kk)} \right) }{u_{\kk}(s)} = \lim_{\kk \to \infty} \frac{u_{\kk}\left( s+ \frac{y}{\gamma(\kk)} \right) }{u_{\kk}(s)} \cdot \frac{1}{1}
        \stackrel{\eqref{eq_u_kk_x}}{=} \\
\lim_{\kk \to \infty}
\frac{u_{\kk}\left( s+ \frac{y}{\gamma(\kk)} \right)}{u_\kk(s)} \cdot
\frac{
u_\kk(s) \cdot \gamma(\kk) \cdot \gt_{\kk}(s) \cdot \psi(s)  }{
 u_\kk\left(s+ \frac{y}{\gamma(\kk)}\right) \cdot \gamma(\kk) \cdot \gt_{\kk}\left(s+ \frac{y}{\gamma(\kk)}\right) \cdot \psi\left(s+ \frac{y}{\gamma(\kk)}\right) }
        \\ =
        \lim_{\kk \to \infty} \frac{\gt_\kk(s)}{\gt_\kk\left( s+ \frac{y}{\gamma(\kk)} \right)} \cdot
 \frac       {  \psi(s)  }{  \psi\left( s+ \frac{y}{\gamma(\kk)} \right)}
        \stackrel{\eqref{psi_exp}}{=} e^{\psi(s)\cdot y}.
    \end{multline}
\end{corollary}

Recall from below \eqref{u_density}
that $V_\kk(x,y):=u_\kk^{-1}(u_\kk(x) + u_\kk(y))$.
\begin{lemma}[Estimating $V_\kk$]\label{lemma_ugly_log} For any $s \in (0,1)$ and $t \in \mathbb{R}$ we have
    \begin{align}\label{ugly_logarithm}
       \lim_{\kk \to \infty} \gamma(\kk) \cdot \left(V_\kk\Big(s+\frac{t}{\gamma(\kk)}, s - \frac{t}{\gamma(\kk)}\Big) - s\right)=  \frac{\ln(e^{\psi(s)t} + e^{-\psi(s)t})}{\psi(s)}.
    \end{align}
    \begin{proof} First we show that
  $  \gamma(\kk) \cdot \left(V_\kk\Big(s+\frac{t}{\gamma(\kk)}, s - \frac{t}{\gamma(\kk)}\Big) - s\right) \leqslant  \frac{\ln(e^{\psi(s)t} + e^{-\psi(s)t})}{\psi(s)} +\delta  $ holds for any $\delta>0$ if $\kk$ is big enough. By $V_\kk(x,y)=u_\kk^{-1}(u_\kk(x) + u_\kk(y))$ and the strict monotonicity of $u_\kk(\cdot)$ we only need to check that if $\kk$ is large enough then we have
  \begin{equation}\label{uuu_delta_ineq}
  u_{\kk}\left( s + \frac{t}{\gamma(\kk)}\right) + u_{\kk}\left( s - \frac{t}{\gamma(\kk)} \right) \leqslant
  u_{\kk}\left( s + \frac{1}{\gamma(\kk)}\cdot \left(\frac{\ln(e^{\psi(s)t} + e^{-\psi(s)t})}{\psi(s)} +\delta \right) \right).
  \end{equation}
Dividing both sides of \eqref{uuu_delta_ineq} by $u_\kk(s)$, letting $\kk \to \infty$ and using Corollary \ref{cor_uu_frac_lim} three times we obtain that \eqref{uuu_delta_ineq} indeed holds is $\kk$ is big enough. The proof of the fact that for any $\delta>0$ the inequality $  \gamma(\kk) \cdot \left(V_\kk\Big(s+\frac{t}{\gamma(\kk)}, s - \frac{t}{\gamma(\kk)}\Big) - s\right) \geq  \frac{\ln(e^{\psi(s)t} + e^{-\psi(s)t})}{\psi(s)} -\delta  $ holds for big enough $\kk$ is analogous and we omit it.  Since $\delta>0$ was arbitrary, we obtain
\eqref{ugly_logarithm}.
    \end{proof}
\end{lemma}
\begin{corollary} For any $s \in (0,1)$ and $t \in \mathbb{R}$ we have
\begin{multline}\label{gt_V_pm_rec}
    \lim_{\kk \to \infty} \frac{\gt_\kk(s)}{\gt_\kk\left( V_\kk\Big(s+\frac{t}{\gamma(\kk)}, s - \frac{t}{\gamma(\kk)}\Big) \right)} = \\
 \lim_{\kk \to \infty} \frac{\gt_\kk(s)}
 {\gt_\kk\left( s+ \frac{1}{\gamma(\kk)} \frac{\ln(e^{\psi(s)t} + e^{-\psi(s)t})}{\psi(s)} \right)} \cdot \frac{\gt_\kk\left( s+ \frac{1}{\gamma(\kk)} \frac{\ln(e^{\psi(s)t} + e^{-\psi(s)t})}{\psi(s)} \right)}
 {\gt_\kk\left( V_\kk\Big(s+\frac{t}{\gamma(\kk)}, s - \frac{t}{\gamma(\kk)}\Big) \right)}
        \stackrel{\eqref{psi_exp}, \eqref{ugly_logarithm}}{=} \\
    \lim_{\kk \to \infty}
  \frac{\gt_\kk(s)}{\gt_\kk\left( s+ \frac{1}{\gamma(\kk)} \frac{\ln(e^{\psi(s)t} + e^{-\psi(s)t})}{\psi(s)} \right)} \cdot 1 \stackrel{\eqref{psi_exp}}{=}  e^{\psi(s)t}+e^{-\psi(s)t}.
\end{multline}
\end{corollary}

\begin{proof}[Proof of Theorem \ref{logistic_limit_thm}] Let us fix $s \in (0,1)$ and $t \in \mathbb{R}$. We only need to prove \eqref{limit_thm_pdf}.
\begin{multline}
\lim_{\kk \to \infty} h_{\kk}(s,t)\stackrel{\eqref{u_density}}{=}
\lim_{\kk \to \infty}
 \frac{-1}{\gamma(\kk)} \cdot \frac{((\gt_\kk)^2)'\left( V_\kk( s+\frac{t}{\gamma(\kk)}, s-\frac{t}{\gamma(\kk)} ) \right)}{\gt_\kk(s+\frac{t}{\gamma(\kk)})\cdot \gt_\kk(s-\frac{t}{\gamma(\kk)})}
 \stackrel{\eqref{psi_exp}}{=} \\
 \lim_{\kk \to \infty}
 \frac{-1}{\gamma(\kk)} \cdot \frac{((\gt_\kk)^2)'\left( V_\kk( s+\frac{t}{\gamma(\kk)}, s-\frac{t}{\gamma(\kk)} ) \right)}{ (\gt_\kk(s))^2}\stackrel{\eqref{gt_V_pm_rec}}{=} \\ \lim_{\kk \to \infty}
 \frac{-1}{\gamma(\kk)} \cdot \frac{((\gt_\kk)^2)'\left( V_\kk( s+\frac{t}{\gamma(\kk)}, s-\frac{t}{\gamma(\kk)} ) \right)}{ \left(\gt_\kk\left( V_\kk( s+\frac{t}{\gamma(\kk)}, s-\frac{t}{\gamma(\kk)} ) \right)\right)^2 \cdot (e^{\psi(s)\cdot t}+e^{-\psi(s)\cdot t})^2} \stackrel{\eqref{def_psi_k}}{=} \\
 \lim_{\kk \to \infty} \frac{2 \cdot \psi_\kk\left(  V_\kk( s+\frac{t}{\gamma(\kk)}, s-\frac{t}{\gamma(\kk)} ) \right)}{(e^{\psi(s)\cdot t}+e^{-\psi(s)\cdot t})^2} \stackrel{(*)}{=}
 \frac{2 \cdot \psi(s) }{(e^{\psi(s)\cdot t}+e^{-\psi(s)\cdot t})^2},
\end{multline}
where $(*)$ follows from parts \eqref{ass_psik_unif_equico} and \eqref{ass_gammak_unif_conv} of Assumption \ref{assumption_kk} and Lemma \ref{lemma_ugly_log}. The proof of
\eqref{limit_thm_pdf} and thus the proof of Theorem \ref{logistic_limit_thm} is complete.
\end{proof}

\section{Symmetry}
\label{section_symm}

Recall the notation of $F_\sigma$ and $V_\sigma$ from Definition \ref{def_permutation_permuton}.
\begin{proposition}[Product form]\label{distribution_prop}
Let $\sigma_n \sim \mathrm{PERM}(\gt,n)$.
For every $\pi \in S_n$ we have
\begin{align}\label{distribution}
    \mathbb{P}(\sigma_n = \pi) =
    \prod_{i=1}^{n} \frac{\gt\left(V_{\pi}\left( \frac{i-1}{n}, \frac{\pi(i)-1}{n}  \right)  \right) - \gt\left( V_{\pi}\left( \frac{i-1}{n}, \frac{\pi(i)-1}{n}  \right) +\frac{1}{n} \right)}{\gt(\frac{i-1}{n})}.
\end{align}
\end{proposition}
\begin{proof} Let us fix $\pi \in S_n$. Recalling Definition \ref{def_gen_perm_model}, for every $i \in [n]$ we have
\begin{multline}\label{condprob_pi}
    \mathbb{P} \left(\sigma_n(i) = \pi(i) \: \big| \: \sigma_n(1) = \pi(1), \ldots, \sigma_n(i-1) = \pi(i-1) \right)
    \stackrel{(*)}{=} \\
    \mathbb{P}\left(\Ct^{(n),i} = (\pi(i)-1) -n \cdot F_\pi(i-1,\pi(i)-1)+1 \right)\stackrel{\eqref{nu_prob_def_eq}, \eqref{cup_empirical_cdf}}{=} \\
    \frac{\gt\left(V_{\pi}\left( \frac{i-1}{n}, \frac{\pi(i)-1}{n}  \right)  \right) - \gt\left( V_{\pi}\left( \frac{i-1}{n}, \frac{\pi(i)-1}{n}  \right) +\frac{1}{n} \right)}{\gt(\frac{i-1}{n})},
\end{multline}
where $(*)$ follows from the fact that the number of cards  below $\pi(i)$ in   the set $D^{(n)}_i$ of cards remaining in the old deck before the $i$'th step is equal to  $(\pi(i)-1) -n \cdot F_\pi(i-1,\pi(i)-1)$, cf.\ \eqref{empirical_cdf}.
The proof of \eqref{distribution} follows from \eqref{condprob_pi} by the chain rule.
\end{proof}
\begin{proposition}[Symmetry]\label{inverse_cor}
If $\sigma_n \sim \mathrm{PERM}(\gt,n)$ then $\sigma^{-1}_n \sim \mathrm{PERM}(\gt,n)$.
\end{proposition}
\begin{proof}
 We only need to check that for every $\pi \in S_n$ we have $
        \mathbb{P}(\sigma_n = \pi^{-1}) = \mathbb{P}(\sigma_n = \pi)$.
First observe that $V_{\pi^{-1}}\left( x, y  \right) \equiv V_{\pi}\left( y, x  \right)$ follows from
\eqref{cup_empirical_cdf} and the fact that the matrix of the permutation $\pi^{-1}$ is the transpose of the matrix of $\pi$. Using this and $\pi(\pi^{-1}((i))=i$ we obtain
$V_{\pi^{-1}}\left( \frac{i-1}{n}, \frac{\pi^{-1}(i)-1}{n}  \right)=V_{\pi}\left( \frac{\pi^{-1}(i)-1}{n}, \frac{\pi(\pi^{-1}((i))-1}{n}  \right)$, thus the numerator of the $i$th term in the product form of $\mathbb{P}(\sigma_n = \pi^{-1})$ is the same as
the numerator of the $\pi^{-1}(i)$th term in the product form of $\mathbb{P}(\sigma_n = \pi)$ (cf.\ \eqref{distribution}). The product of the numerators in the product forms of $
        \mathbb{P}(\sigma_n = \pi^{-1})$ and $ \mathbb{P}(\sigma_n = \pi)$  contain the same terms in different order, while the denominators are the same, thus
    we indeed have $
        \mathbb{P}(\sigma_n = \pi^{-1}) = \mathbb{P}(\sigma_n = \pi)$.
\end{proof}

\section*{Acknowledgements}
We thank P\'eter Csikv\'ari for his contributions to the proof of Proposition \ref{inverse_cor}   and Mikl\'os Ab\'ert for useful discussions. This work is partially supported by the ERC Synergy under Grant No. 810115 - DYNASNET and the grant NKFI-FK-142124 of NKFI
(National Research, Development and Innovation Office). 

\begin{thebibliography}{99}

\bibitem{runsort}
Alon, N., Defant, C., \& Kravitz, N. (2022). The runsort permuton. \emph{Advances in Applied Mathematics}, 139, 102361.

\bibitem{separable}
Bassino, F., Bouvel, M., F\'eray, V., Gerin, L., \& Pierrot, A. (2018). The Brownian limit of separable permutations. \emph{The Annals of Probability}, 46(4), 2134-2189.


\bibitem{biased_permutations} Bouvel, M., Nicaud, C., \& Pivoteau, C. (2024). Record-biased permutations and their permuton limit. \emph{arXiv preprint} arXiv:2409.01692.

\bibitem{Dubach} Dubach, V. (2024). Classical patterns in Mallows permutations. \newline \emph{arXiv preprint} arXiv:2410.17228.




\bibitem{bclogistic}
Bufetov, A., \& Chen, K. (2024). Mallows Product Measure. \emph{Electronic Journal of Probability,}  29: 1--33.



\bibitem{Bufetov_clt} Bufetov, A., \& Chen, K. (2024). Local central limit theorem for Mallows measure. \emph{arXiv preprint} arXiv:2409.10415.


\bibitem{GP}
Gladkich, A., \& Peled, R. (2018). On the cycle structure of Mallows permutations. \emph{The Annals of Probability} 46(2): 1114-1169.

\bibitem{GO} Gnedin, A., \&  Olshanski, G. (2010). "$q$-Exchangeability via
quasi-invariance." \emph{The Annals of Probability} 38(6), 2103-2135.

\bibitem{the_paper} Hoppen, C., Kohayakawa, Y., Moreira, C. G., R\'ath, B., \& Sampaio, R. M. (2013). Limits of permutation sequences. \emph{Journal of Combinatorial Theory, Series B}, 103(1), 93-113.



\bibitem{Mallows} Mallows, C.\ L. (1957). Non-null ranking models. I. \emph{Biometrika} 44.1/2  114-130.

\bibitem{starr}
Starr, S. (2009). Thermodynamic limit for the Mallows model on $S_n$. \emph{Journal of mathematical physics} 50.9.

\bibitem{Travers} Travers, N. (2015). Inversions and longest increasing subsequence for $k$-card-minimum random permutations. \emph{Electronic Journal of Probability} 20 1 -- 27.


\end{thebibliography}
\end{document}